\documentclass[12pt]{article}
\usepackage{latexsym,amsfonts,amssymb,mathrsfs,float}
\usepackage{dsfont}
\usepackage{amsthm}
\usepackage[numbers]{natbib}
\usepackage{amsmath}
\usepackage{indentfirst}

\usepackage[dvips]{color}
\usepackage{colordvi,multicol}
\allowdisplaybreaks

\newtheorem{thm}{Theorem}[section]
\newtheorem{cor}[thm]{Corollary}
\newtheorem{lem}[thm]{Lemma}
\newtheorem{pro}[thm]{Proposition}

\numberwithin{equation}{section}

\begin{document}

\title{\bf Three Favorite Edges Occurs Infinitely Often for One-Dimensional Simple Random Walk}
 \author{ Chen-Xu Hao$^1$, Ze-Chun Hu$^1$, Ting Ma$^1$ and Renming Song$^2$\\ \\
 {\small $^1$College of Mathematics, Sichuan  University,
 Chengdu 610065, China}\\
 {\small 476924193@qq.com; zchu@scu.edu.cn; matingting2008@scu.edu.cn }\\ \\
  {\small $^2$Department of Mathematics,
University of Illinois, Urbana, IL 61801, USA}\\
 {\small rsong@illinois.edu}}

 \date{}
\maketitle

 \noindent{\bf Abstract:}\quad For a one-dimensional simple symmetric random walk $(S_n)$,
an edge $x$ (between points $x-1$ and $x$) is called a favorite edge at time $n$ if its local time at $n$ achieves the maximum among all edges.
In this paper, we show that with probability 1 three favorite edges occurs infinitely often. Our work is inspired by T\'{o}th and Werner [Combin. Probab. Comput. {\bf 6} (1997) 359-369], and Ding and Shen [Ann. Probab. {\bf 46} (2018) 2545-2561],
disproves a conjecture mentioned in Remark 1 on page 368 of T\'{o}th and Werner [Combin. Probab. Comput. {\bf 6} (1997) 359-369].

\smallskip


\smallskip


\section{Introduction}

Let $(S_n)_{n\in\mathbb{N}}$
be a one-dimensional simple symmetric random walk with $S_0=0$. Following T\'{o}th and Werner \cite{TW97}, we define for any $i\geq 1$,
\begin{eqnarray*}
\widetilde{S}_i:=\dfrac{S_i+S_{i-1}+1}{2},
\end{eqnarray*}
which characterizes the edge of $i$-th jump (edge $x$ is between points $x-1$ and $x$), and also define the ``local time on the edge $x$ at time $n$" as follows:
\begin{eqnarray}\label{local-time-of-edge}
L(x,n):=\#\left\{1\leq j\leq n: \widetilde{S}_j=x\right\}.
\end{eqnarray}
Hereafter, $\# D$ denotes the cardinality of the set $D$.
An edge $x$ is called a favorite (or most visited) edge of the random walk at time $n$ if
$$
L(x,n)=\sup_{y\in{\mathbb{Z}}}L(y,n).
$$
The set of favorite edges of the random walk at time $n$ is denoted by $\mathcal{K}(n)$. $(\mathcal{K}(n))_{n\ge 1}$ is called the favorite edge process of the one-dimensional simple symmetric random walk.
We say that {\it three favorite edges} occurs at time $n$ if $\#\mathcal{K}(n)=3$.

\begin{thm}\label{mainthm}
For one-dimensional simple symmetric random walk, with probability $1$, three favorite edges occurs infinitely often.
\end{thm}

Theorem \ref{mainthm} complements the result in \cite{TW97} which showed that eventually there are no more than three favorite edges, and disproves a conjecture
mentioned in \cite[Remark 1 on page 368]{TW97}.

For the related problem of the number of favorite sites of one-dimensional simple symmetric random walks, there are many more references (see Shi and T\'{o}th \cite{ST00} for an overview).
This problem was posed by Erd\H{o}s and R\'{e}v\'{e}sz \cite{ER84, ER87, ER91, Re90}. T\'{o}th \cite{To01} proved that there are no more than three favorite sites eventually. Ding and Shen \cite{DS18} proved that with probability 1 three favorite sites occurs infinitely often.

Besides the number of favorite sites, a series of papers focus on the asymptotic behavior of favorite sites, see \cite{BG85, CRS00, CS98,  LS04}.
In addition, there are a number of papers
on favorite sites for other processes
including Brownian motion, symmetric stable process, L\'{e}vy processes, random walks in random environments  and so on, see \cite{BES00, CDH18, Ei89, Ei97, EK02,  HS00, HS15, KL95, Le97, LXY19, Ma01}.

For papers on favorite sites of simple random walks in  higher dimensions,
we refer to \cite{Ab15, De05, DPRZ01, Ok16}.

Our proof of Theorem \ref{mainthm} is inspired  by \cite{DS18}, which in turn was inspired by \cite{To01, TW97}. Following \cite{To01}, we define the number of upcrossings and downcrossings
of the site $x$ by time $n$ to be
\begin{eqnarray*}
\xi_U(x,n):=\#\{0<k\leq n:S_k=x,S_{k-1}=x-1\},\\
\xi_D(x,n):=\#\{0<k\leq n:S_k=x,S_{k-1}=x+1\}
\end{eqnarray*}
respectively. It is easy to check that
\begin{eqnarray}\label{1.1}
\xi_U(x,n)-\xi_D(x-1,n)=\textbf{1}_{\{0<x\leq S_n\}}-\textbf{1}_{\{S_n<x\leq 0\}}.
\end{eqnarray}
Using (\ref{local-time-of-edge}) and  (\ref{1.1}), we can easily get (see
\cite{To01})
\begin{eqnarray}\label{UDcrossings}
L(x,n)&=&\#\left\{0<j\leq n:\widetilde{S}_j=x\right\}\nonumber\\
&=&\#\left\{0<j\leq n:S_j=x,S_{j-1}=x-1\right\}+\#\left\{0<j\leq n:S_j=x-1,S_{j-1}=x\right\}\nonumber\\
&=&\xi_U(x,n)+\xi_D(x-1,n)\nonumber\\
&=&2\xi_U(x,n)+\textbf{1}_{\{S_n<x\leq 0\}}-\textbf{1}_{\{0<x\leq S_n\}}\nonumber\\
&=&2\xi_D(x-1,n)+\textbf{1}_{\{0<x\leq S_n\}}-\textbf{1}_{\{S_n<x\leq 0\}}.
\end{eqnarray}

For $r\geq 1$, let $f(r)$ be the (possibly infinite) number of times when there are exactly $r$ favorite edges:
\begin{eqnarray*}
f(r):=\#\left\{n\geq 1:\#\mathcal{K}(n)=r\right\}.
\end{eqnarray*}

We will show that
\begin{eqnarray}\label{f3 infinity}
f(3)=\infty\ \ \mbox{with probability}\ 1,
\end{eqnarray}
which implies Theorem \ref{mainthm}. The idea for the proof comes from \cite{DS18}.
In the proof, we need  a transience result for the favorite edge process of one-dimensional simple symmetric random walks. To prove this transience result, we establish an invariance principle for one-side local times of random walks.

We mention in passing that, similar to \cite{DS18, To01},  we could have defined
the following quantity related $f(r)$:
$$
\widetilde{f}(r):=\#\left\{n\geq 1: \widetilde{S}_n\in \mathcal{K}(n)
\mbox{ and }\#\mathcal{K}(n)=r\right\},
$$
which is the number of times at which a new favorite
edge appears, tied with $r-1$
other favorite edges. Using the recurrence of
one-dimensional simple symmetric random walk, one can easily see that
$f(3)=\infty$ a.s. is equivalent to $\widetilde{f}(3)=\infty$ a.s.
See for instance, the paragraph below \cite[Theorem 1.1]{TW97}.

The rest of the paper is organized as follows. In Section 2, we
establish an invariance principle for one-side local times of random walks,
and then use it to prove the transience of the favorite edge process.
In Section 3, we set up the framework for our proof for Theorem \ref{mainthm}. In Section 4, we give the proof of Theorem \ref{mainthm}. We emphasize that the idea for the proof of Theorem \ref{mainthm} comes from \cite{DS18}, and
our main contributions are the invariance principle for one-side local times and the transience of the favorite edge process. In final section, we give a remark.

\section{The transience of the favorite edge process}

In this section, we study the transience of
the favorite edge process
and  prove the following:

\begin{thm}\label{thm-2.1}
For any $\gamma>11$,  we have
\begin{eqnarray}\label{thm-2.1-a}
\liminf_{n\to\infty}\frac{\tilde{U}(n)}{\sqrt{n}(\log n)^{-\gamma}}=\infty
\qquad a.s.,
\end{eqnarray}
where  $\tilde{U}(n):=\min\{|x|:x\in\mathcal{K}(n)\}$.
\end{thm}

To prove the above theorem, we will establish an invariance principle
for one-side local times of random walks in Section 2.2. The proof of Theorem \ref{thm-2.1} will be given in Section 2.3. In Section 2.1, we give a  brief introduction for invariance principle for (two-side) local times.

\subsection{Invariance principle for local times}

Recall that the (site) local time process of the random walk $(S_n,n\geq 0)$ is defined by
\begin{eqnarray*}
\xi(x,n):=\#\left\{k:1\leq k\leq n,S_k=x\right\},\ x\in\mathbb{Z},\ n\geq 1.
\end{eqnarray*}
Define
$
\xi^*(n):=\sup_{x\in\mathbb{Z}}\xi(x,n).
$

Let $(W(t))_{t\geq 0}$ be a one-dimensional standard  Brownian motion (Wiener process).
Recall that the local time process $(\eta(x,t))_{t\ge 0, x\in \mathbb{R}}$ of $(W(t))_{t\geq 0}$ is defined by
\begin{eqnarray}\label{equ-2.2-0}
\eta(x,t):=\lim_{\varepsilon\downarrow  0}\frac{1}{2\varepsilon}\int_0^t\textbf{1}_{(x-\varepsilon,x+\varepsilon)}(W(s))ds=
\lim_{\varepsilon\downarrow  0}\frac{1}{\varepsilon}\int_0^t\textbf{1}_{[x,x+\varepsilon)}(W(s))ds.
\end{eqnarray}

R\'{e}v\'{e}sz \cite{Re81} established
the following strong invariance principle with a rate of convergence: On a rich enough probability space, as $n\to\infty$,
\begin{eqnarray}\label{equ-2.2}
\sup_{x\in \mathbb{Z}}|\xi(x,n)-\eta(x,n)|=o(n^{\frac{1}{4}+\varepsilon})\quad a.s.
\end{eqnarray}
for any $\varepsilon>0$.    Cs\"{o}rg\H{o} and Horv\'{a}th
\cite[Theorems 1 and 2]{CH89}
showed that the R\'{e}v\'{e}sz's result  is the best possible.

\subsection{Invariance principle for one-side local times}

Define the one-side local time of the Wiener process by
$$
\eta_R(x,t):=\lim_{\varepsilon\downarrow 0}\frac{1}{2\varepsilon}\int_0^t\textbf{1}_{[x,x+\varepsilon)}\big(W(s)\big)ds.
$$
Then by (\ref{equ-2.2-0}), we get that
 for any $x\in \mathbb{R}$ and $t\ge 0$,
\begin{eqnarray}\label{2.2-a}
\eta_R(x,t)=\frac{1}{2}\eta(x,t).
\end{eqnarray}

The goal in this subsection is to extend (\ref{equ-2.2}) to the one-side case, that is,  prove the following result.

\begin{thm}\label{One-side Invariance principle}
On a rich enough probability space $(\Omega,\mathcal{F},P)$, one can define a Wiener process $(W(t))_{t\geq 0}$ and one-dimensional simple symmetric random walk $(S_k)_{k\in \mathbb{N}}$ with $S_0=0$, such that for any $\varepsilon>0$, as $n\to\infty$, we have
\begin{eqnarray*}
\sup_{x\in \mathbb{Z}}|\xi_D(x,n)-\eta_R(x,n)|=o(n^{\frac{1}{4}+\varepsilon})\ ~~~~a.s.
\end{eqnarray*}
\end{thm}

Before giving the proof of Theorem \ref{One-side Invariance principle}, we prove the following lemma.

\begin{lem}\label{One-side local time difference}
For any $\varepsilon>0$, we have
\begin{eqnarray}\label{One-side local time lemma00}
\lim_{n\to\infty}\frac{\sup\limits_{x\in\mathbb{Z}}|\xi_D(x+1,n)-\xi_D(x,n)|}
{n^{\frac{1}{4}+\varepsilon}}=0\quad a.s.
\end{eqnarray}
\end{lem}
\noindent{\bf Proof.} The idea of the proof comes from Cs\'{a}ki and R\'{e}v\'{e}sz \cite[Lemma 5]{CR83}, where the corresponding problem for
(two-side) local times was considered in a more general setting.
First we show that, for all $m\geq 2$ and all $\delta>0$, there exists
a constant $C$
(which may depends on $m$ but not on $n$) such that
\begin{eqnarray}\label{One-side local time lemma01}
E(|\xi_D(1,n)-\xi_D(0,n)|^m)\leq Cn^{\frac{m}{4}+\delta}.
\end{eqnarray}

Define
\begin{eqnarray}\label{One-side local time lemma01-a}
T_k:=\sum_{i=1}^{k}\left(\xi_D(1,\alpha_i)-\xi_D(1,\alpha_{i-1})-1\right),
\end{eqnarray}
where $\alpha_0=0$, $\alpha_i=\min\{j>\alpha_{i-1}:S_j=0,S_{j-1}=1\},\forall i\geq 1$.
We claim that
\begin{eqnarray}\label{One-side local time lemma02}
E(\xi_D(1,\alpha_1))=1,\quad E(\xi_D^m(1,\alpha_1))<+\infty, \forall m\geq 2.
\end{eqnarray}
To prove this claim, we follow Kesten and Spitzer \cite[Lemma 2]{KS79} and define,
for $x,y\in \mathbb{Z}$, $\tau_0(x):=0,\tau_i(x):=\min\{n>\tau_{i-1}(x):S_n=x,S_{n-1}=x+1\},\forall i\geq 1$. Note that $\alpha_i=\tau_i(0)$. For $j\geq 0$,
\begin{eqnarray*}
M_j(x,y):=\sum_{\tau_j(x)<n\leq\tau_{j+1}(x)}\textbf{1}_{\{S_n=y,S_{n-1}=y+1\}}
\end{eqnarray*}
is the number of downcrossings to $y$ between the $j^{th}$ and $(j+1)^{th}$ downcrossings to $x$.
In particular,  $M_0(0,1)=\xi_D(1,\alpha_1)$. By the strong Markov property, we know that the distribution of
$M_j(0,y)$, $j\ge 0$, is independent of $j$, and the distribution of $M_j(x,y)$, $j\ge 1$, is independent of $j$.
Thus we can
define for $j\geq 1$,
\begin{eqnarray*}
p(x,y)&:=&P(M_j(x,y)\neq 0)\\
&=&P(S_n=y,S_{n-1}=y+1~\mbox{for some}\  n\ \mbox{with~} \tau_j(x)<n\leq\tau_{j+1}(x)).
\end{eqnarray*}

We claim that, for any  $x, y\in \mathbb{Z}$ with $|x-y|=1$, it holds that $p(x,y)=p(y,x)=1/2$. To prove this claim,
it suffices to show that $p(0, 1)=p(1, 0)=1/2$.
When $\tau_1(0)<\infty$, there exists $n_0=\min\{k\ge 0: S_k=1\}<\tau_1(0)$. Then $S_{n_0}=1$, $\{S_{n_0+1}=0\}\subset \{M_0(0, 1)=0\}$
and $\{S_{n_0+1}=2\}\subset \{M_0(0, 1)\neq 0\}$. Since $P(\tau_1(0)<\infty)=1$, we have
\begin{align*}
&p(0, 1)=P(M_0(0, 1)\neq 0, \tau_1(0)<\infty)\\
&=P(S_{n_0+1}=0, M_0(0, 1)\neq 0, \tau_1(0)<\infty)+
P(S_{n_0+1}=2, M_0(0, 1)\neq 0, \tau_1(0)<\infty)\\
&=P(S_{n_0+1}=2)=\frac12.
\end{align*}
Similarly, there exists $n_1=\min\{k\ge 0: S_k=1,S_{k-1}=2\}<\tau_1(1)$, then
\begin{align*}
&p(1, 0)=P(S_{n_1+1}=2, M_1(1, 0)\neq 0)+P(S_{n_1+1}=0, M_1(1, 0)\neq 0)\\
&=P(S_{n_1+1}=0)=\frac12.
\end{align*}

Combining the claim above with the strong Markov property and the fact $S_{\tau_j(x)}=S_{\tau_{j+1}(x)}=x$, we get
\begin{equation*}
P(M_j(0,1)=k)=\left\{
\begin{array}{ll}
1-p(0,1)=\frac{1}{2}, &\mbox{if}\ k=0,\\
p(0,1)[1-p(1,0)]^{k-1}p(1,0)=\frac{1}{2^{k+1}},&\text{if}\ k\geq 1.
\end{array}
\right.
\end{equation*}
It follows that
\begin{eqnarray*}
&&E(\xi_D(1,\alpha_1))=E(M_j(0,1))=\sum_{k=1}^{\infty}k\cdot \frac{1}{2^{k+1}}=1,\\
&&E(\xi_D^m(1,\alpha_1))=E(M_j^m(0,1))=\sum_{k=1}^{\infty}k^m\cdot \frac{1}{2^{k+1}}<+\infty,
\end{eqnarray*}
which implies that (\ref{One-side local time lemma02}) holds.

It follows from the strong Markov property
that $T_k$ is a sum of i.i.d. r.v.'s with mean $0$ and
finite moments of all orders.
Therefore by Chung \cite{Ch51} and the $L^p$-maximal inequality for martingales, we have
\begin{eqnarray}\label{One-side local time lemma03}
E(|T_k|^m)\leq C_1k^{\frac{m}{2}}
\end{eqnarray}
and for any $\delta_1>0$,
\begin{eqnarray}\label{One-side local time lemma04}
E\left(\max_{k\leq n^{\frac{1}{2}+\delta_1}}|T_k|^m\right)\leq
C_2n^{\frac{m}{4}+\frac{m\delta_1}{2}}.
\end{eqnarray}
Note that
\begin{eqnarray*}
\xi_D(1,\alpha_{D_n})-\xi_D(0,n)\leq\xi_D(1,n)-\xi_D(0,n)\leq\xi_D(1,\alpha_{D_{n}+1})-
(\xi_D(0,n)+1)+1,
\end{eqnarray*}
where $D_n=\xi_D(0,n)$. Thus
on the event $\{D_n+1\leq n^{\frac{1}{2}+\delta_1}\}$, we have
\begin{eqnarray*}
|\xi_D(1,n)-\xi_D(0,n)|\leq\max_{1\leq k\leq n^{\frac{1}{2}+\delta_1}}|T_k|+1.
\end{eqnarray*}
Hence
\begin{eqnarray}\label{One-side local time lemma05}
E(|\xi_D(1,n)-\xi_D(0,n)|^m)\leq C E\left(\max_{1\leq k\leq n^{\frac{1}{2}+\delta_1}}(|T_k|+1)^m\right)+n^mP\left(D_n+1\geq n^{\frac{1}{2}+\delta_1}\right).
\end{eqnarray}
Let  $N_n=N(0,n)=\#\{k:0<k\leq n,S_k=0\}$, then according to Kesten and Spitzer \cite{KS79}, $E(N_n^m)=O(n^{\frac{m}{2}})$. Combining this with the fact that $0\leq D_n\leq N_n$, we get that $E(D_n^m)=O(n^{\frac{m}{2}})$.
Then by Markov's inequality,  (\ref{One-side local time lemma04}) and (\ref{One-side local time lemma05}),
letting $\delta_1$ be small enough, we obtain (\ref{One-side local time lemma01}).

By repeating the argument above  and noticing that
$p(-1,0)=p(0,-1)=\frac{1}{2}$, we can obtain the inequality (\ref{One-side local time lemma01}) with $\xi_D(1,n)$ replaced by $\xi_D(-1,n)$.

It is easy to see that for any $x\in \mathbb{Z}$,
\begin{eqnarray}\label{One-side local time lemma06}
E(|\xi_D(x+1,n)-\xi_D(x,n)|^m)\leq Cn^{\frac{m}{4}+\delta},
\end{eqnarray}
where the constant $C$ does not depend on $x$, because $\xi_D(x+1,n)-\xi_D(x,n)$ is stochastically smaller than $\max(|\xi_D(1,n)-\xi_D(0,n)|,|\xi_D(-1,n)-\xi_D(0,n)|)$.

Choosing $m=\frac{2+\delta}{\varepsilon}$, we obtain by Markov's inequality that
\begin{eqnarray*}
P\Big(|\xi_D(x+1,n)-\xi_D(x,n)|>n^{\frac{1}{4}+\varepsilon}\Big)\leq\dfrac{C}{n^2}
\end{eqnarray*}
and hence
\begin{eqnarray*}
P\left(\sup_{|x|<(n\log n)^{\frac{1}{2}}}|\xi_D(x+1,n)-\xi_D(x,n)|>n^{\frac{1}{4}+\varepsilon}\right)\leq\dfrac{2C(\log n)^{\frac{1}{2}}}{n^{\frac{3}{2}}}.
\end{eqnarray*}
Therefore, by the Borel-Cantelli Lemma, we get that
\begin{eqnarray*}
\lim_{n\to+\infty}\dfrac{\sup_{|x|\leq(n\log n)^{\frac{1}{2}}}|\xi_D(x+1,n)-\xi_D(x,n)|}{n^{\frac{1}{4}+\varepsilon}}=0~~~~a.s.
\end{eqnarray*}
The law of the iterated logarithm for $S_n$ implies that $\xi_D(x,n)=0~~a.s.$ for $|x|\geq (n\log n)^{\frac{1}{2}}$ and $n$ sufficiently large, hence (\ref{One-side local time lemma00}) holds.\hfill\fbox

\bigskip

\noindent {\bf Proof of Theorem \ref{One-side Invariance principle}.}
This proof is inspired by R\'{e}v\'{e}sz \cite{Re81}, and Cs\"{o}rg\H{o} and Horv\'{a}th \cite{CH89}.
Let $(W(t))_{t\geq 0}$ be a Wiener process and define $\tau_0:=0,\tau_1:=\inf\{t:t>0,|W(t)|=1\},\tau_n:=\inf\{t:t>\tau_{n-1},|W(t)-W(\tau_{n-1})|=1\}, \forall n\geq 2$. Then $X_i=W(\tau_i)-W(\tau_{i-1})(i\geq 1)$ are i.i.d. r.v.'s with $P\{X_i=1\}=P\{X_i=-1\}=1/2$, and  $\tau_i-\tau_{i-1}(i\geq 1)$ are i.i.d. r.v.'s with $E(\tau_i-\tau_{i-1})=1$ and $E(\tau_i-\tau_{i-1})^2<\infty$. Put  $\sigma^2=E(\tau_1-1)^2$.
Define $S_k=X_1+\cdots+X_k=W(\tau_k)$.

Let $a_i(x)=\eta(x,\tau_{v(i)+1})-\eta(x,\tau_{v(i)-1}),b_i(x)=\eta_R(x,\tau_{v(i)+1})-\eta_R(x,\tau_{v(i)-1})(i\in\mathbb{Z}^+)$, where
$v(1)=\min\{k\geq 0,W(\tau_k)(=S_k)=x\}, v(n)=\min\{k>v(n-1),W(\tau_k)(=S_k)=x\},\forall n\geq 2$.
Then by (\ref{2.2-a}), we have
\begin{eqnarray}\label{proof-thm-2.2-a}
b_i(x)=\frac{a_i(x)}{2}.
\end{eqnarray}

Kesten \cite{Ke65} showed
\begin{eqnarray}\label{LIL-65}
\limsup_{n\to\infty}(2n\log\log n)^{-\frac{1}{2}}\cdot\sup_{x\in\mathbb{Z}}\xi(x,n)=1~~~~~a.s.
\end{eqnarray}
Cs\"{o}rg\H{o} and Horv\'{a}th \cite[(2.7)]{CH89}  says that
\begin{eqnarray*}
\max_{-n\leq x\leq n}\left|\sum_{i=1}^{\xi(x,n)}a_i(x)-\eta(x,\tau_n)\right|=O(\log n)\quad a.s.,
\end{eqnarray*}
which together with (\ref{2.2-a}) and (\ref{proof-thm-2.2-a}) implies that
\begin{eqnarray}\label{One-side IP Proof 01}
\max_{-n\leq x\leq n}\left|\sum_{i=1}^{\xi(x,n)}b_i(x)-\eta_R(x,\tau_n)\right|=O(\log n)~~~~~a.s.
\end{eqnarray}
By \cite[(2.8)]{CH89} and (\ref{proof-thm-2.2-a}), we have
\begin{eqnarray}\label{tool01}
P\left\{\max_{-k\leq x\leq k}\left|\sum_{i=1}^{k}\left(b_i(x)-\frac{1}{2}\right)\right|>C_1(k\log k)^{\frac{1}{2}}\right\}\leq C_2k^{-2}.
\end{eqnarray}
Combining this with the Borel-Cantelli Lemma and then using  (\ref{LIL-65}), we get that
\begin{eqnarray}\label{One-side IP Proof 02}
\limsup_{n\to\infty}\frac{\max\limits_{-n\leq x\leq n}|\sum_{i=1}^{\xi(x,n)}b_i(x)-\frac{\xi(x,n)}{2}|}{n^{\frac{1}{4}}\cdot(\log n)^{\frac{1}{2}}\cdot(\log\log n)^{\frac{1}{4}}}\leq C_3~~~~~a.s.
\end{eqnarray}
Since
$\tau_i-\tau_{i-1},i\geq 1,$ are i.i.d. r.v.'s with $E\tau_1=1$ and $\sigma^2=E(\tau_1-1)^2<\infty$, by the law of iterated logarithm, we have
\begin{eqnarray}\label{LIL-8983}
\limsup_{n\to\infty}\frac{|\tau_n-n|}{\sqrt{2\sigma^2n\log \log n}}=1~~~~~a.s.
\end{eqnarray}
By \cite[(2.11)]{CH89} and (\ref{2.2-a}), we have
\begin{eqnarray*}
\limsup_{n\to\infty}\dfrac{\sup\limits_x|\eta_R(x,n\pm g(n))-\eta_R(x,n)|}{n^{\frac{1}{4}}\cdot(\log n)^{\frac{1}{2}}\cdot(\log\log n)^{\frac{1}{4}}}\leq C_4\quad a.s.,
\end{eqnarray*}
where $g(n)=(4\sigma^2n\log\log n)^{\frac{1}{2}}$.
(Note that there is a minor typo in the line below \cite[(2.11)]{CH89}, the $-\frac12$ there should be $\frac12$.) Thus
by (\ref{LIL-8983}) we have
\begin{eqnarray}\label{One-side IP Proof 03}
\limsup_{n\to\infty}\dfrac{\sup\limits_x|\eta_R(x,\tau_n)-\eta_R(x,n)|}{n^{\frac{1}{4}}\cdot(\log n)^{\frac{1}{2}}\cdot(\log\log n)^{\frac{1}{4}}}\leq C_5~~~~~a.s.
\end{eqnarray}
Note that $\xi_D(x,n)=0$ if $|x|>n$ and that $\lim_{n\to\infty}\sup_{|x|>n}\eta_R(x,n)=0\ a.s.$ Thus
\begin{eqnarray}\label{One-side IP Proof 04}
\lim_{n\to\infty}\sup_{|x|>n}\dfrac{\xi_D(x,n)-\eta_R(x,n)}{n^{\frac{1}{4}+\varepsilon}}=0~~~~~a.s.
\end{eqnarray}
Now by (\ref{One-side IP Proof 01}), (\ref{One-side IP Proof 02}), (\ref{One-side IP Proof 03}), (\ref{One-side IP Proof 04}) and Lemma \ref{One-side local time difference}, we obtain
\begin{eqnarray*}
&&\limsup_{n\to\infty}\dfrac{\sup\limits_x|\xi_D(x,n)-\eta_R(x,n)|}{n^{\frac{1}{4}+\varepsilon}}\\
&&\leq\limsup_{n\to\infty}\dfrac{\sup\limits_{x\leq n}|\xi_D(x,n)-\frac{\xi(x,n)}{2}|+\sup\limits_{x\leq n}|\frac{\xi(x,n)}{2}-\sum_{i=1}^{\xi(x,n)}b_i(x)|}{n^{\frac{1}{4}+\varepsilon}}\\
&&\quad+\limsup_{n\to\infty}\dfrac{\sup\limits_{x\leq n}|\sum_{i=1}^{\xi(x,n)}b_i(x)-\eta_R(x,\tau_n)|+\sup\limits_{x\leq n}|\eta_R(x,\tau_n)-\eta_R(x,n)|}
{n^{\frac{1}{4}+\varepsilon}}\\
&&\quad+\lim_{n\to\infty}\sup\limits_{|x|>n}\dfrac{|\xi_D(x,n)-\eta_R(x,n)|}{n^{\frac{1}{4}+\varepsilon}}\\
&&\leq\limsup_{n\to\infty}\dfrac{\sup\limits_{x\leq n}\frac{|\xi_D(x,n)-\xi_D(x-1,n)|+1}{2}+\sup\limits_{x\leq n}|\frac{\xi(x,n)}{2}-\sum_{i=1}^{\xi(x,n)}b_i(x)|}{n^{\frac{1}{4}+\varepsilon}}\\
&&\quad+\limsup_{n\to\infty}\dfrac{\sup\limits_{x\leq n}|\sum_{i=1}^{\xi(x,n)}b_i(x)-\eta_R(x,\tau_n)|+\sup\limits_{x\leq n}|\eta_R(x,\tau_n)-\eta_R(x,n)|}
{n^{\frac{1}{4}+\varepsilon}}\\
&&\quad+\lim_{n\to\infty}\sup\limits_{|x|>n}\dfrac{|\xi_D(x,n)-\eta_R(x,n)|}{n^{\frac{1}{4}+\varepsilon}}\\
&&=0\ a.s.,
\end{eqnarray*}
where we used the following fact
$$
\left|\xi_D(x,n)-\frac{\xi(x,n)}{2}\right|=\frac{|\xi_D(x,n)-\xi_U(x,n)|}{2}
\leq \frac{|\xi_D(x,n)-\xi_D(x-1,n)|+1}{2}.
$$
 The proof is complete.\hfill\fbox

\subsection{Proof of Theorem \ref{thm-2.1}}

Set
\begin{eqnarray*}
&&\eta_R^*(n):=\sup_{x\in\mathbb{R}}\eta_R(x,n),\quad \eta^*(n):=\sup_{x\in\mathbb{R}}\eta(x,n),\\
&&\xi_D^*(n):=\sup_{x\in\mathbb{Z}}\xi_D(x,n),\quad T_r:=\inf\{n>0:\eta_R(0,n)\geq r\},\\
&&I_R(h,n):=\sup_{|x|\leq h}\eta_R(x,n),
\end{eqnarray*}
and
\begin{eqnarray*}
\mathcal{V}(n):=\{x\in\mathbb{Z}:\eta_R(x,n)=\eta_R^*(n)\},\quad \mathcal{U}(n):=\{x\in\mathbb{Z}:\xi_D(x,n)=\xi_D^*(n)\}.
\end{eqnarray*}

It is easy to see
that Theorem \ref{thm-2.1} is a consequence of the following two propositions.

\begin{pro}\label{pro-2.4}
If $x\in \mathcal{K}(n)$, then $x-1\in \mathcal{U}(n)$.
\end{pro}

\begin{pro}\label{pro-2.5}
For any $\gamma>11$,  we have
\begin{eqnarray}\label{pro-2.5-a}
\liminf_{n\to\infty}\frac{U(n)}{\sqrt{n}(\log n)^{-\gamma}}=\infty
\qquad a.s.,
\end{eqnarray}
where $U(n):=\min\{|x|:x\in\mathcal{U}(n)\}$.
\end{pro}

\subsubsection{Proof of Proposition \ref{pro-2.4}.}

Assume that $x\in \mathcal{K}(n)$, that is, $x$ is a favorite edge at time $n$. We want to prove that $x-1\in \mathcal{U}(n)$. Let $\xi_D(x-1,n)=h$ for some nonnegative integer $h$. Then by (\ref{UDcrossings}), we know that $L(x,n)\in \{2h-1,2h,2h+1\}$.
We will prove $x-1\in \mathcal{U}(n)$ by contradiction. Suppose that $x-1\notin \mathcal{U}(n)$. Then for any $y\in \mathcal{U}(n)$, we have $\xi_D(y,n)\geq h+1$.

{\it Case 1.} $L(x,n)=2h-1$ or $2h$.  By (\ref{UDcrossings}), we get that for any $y\in \mathcal{U}(n)$,
$$L(y+1,n)\geq 2\xi_D(y,n)-1\geq 2(h+1)-1=2h+1.
$$
This implies that $x\notin \mathcal{K}(n)$, which is a contradiction.

{\it Case 2.} $L(x,n)=2h+1$.  Now by (\ref{UDcrossings}), we know that $0<x\leq S_n$.
For any $y\in \mathcal{U}(n)$, we have the following two subcases:

{\it Case 2.1.} $y\leq 0$ or $y>S_n$. Now by (\ref{UDcrossings}), we get that
$$
L(y+1,n)=2\xi_D(y,n)\geq 2(h+1)=2h+2.
$$
This implies that  $x\notin \mathcal{K}(n)$, which is a contradiction.

{\it Case 2.2.} $0<y\leq S_n$. Now by (\ref{UDcrossings}), we get that
$$
L(y+1,n)=2\xi_D(y,n)+1\geq 2(h+1)+1=2h+3.
$$
This implies that  $x\notin \mathcal{K}(n)$, which is a contradiction.

Thus we must have $x-1\in \mathcal{U}(n)$. The proof is complete.\hfill\fbox

\subsubsection{Proof of Proposition \ref{pro-2.5}.}

To prove Proposition \ref{pro-2.5}, we need several lemmas. By \cite[(5.1)]{BG85} and (\ref{2.2-a}), we have

\begin{lem}\label{inf-edge}
For any  $\alpha>5$ and $\varepsilon>0$, there exists $n_0$ such that,
with probability one, we have
\begin{eqnarray*}
\eta_R^*(n)>I_R\left(\frac{\sqrt{n}}{(\log n)^{2\alpha+1+\varepsilon}},n\right)+\frac{1}{2}n^{\frac{1}{2}-\varepsilon},
\quad n\ge n_0.
\end{eqnarray*}
\end{lem}

By \cite[Lemma 5.3]{BG85} and (\ref{2.2-a}), we have
\begin{lem}\label{edge-tool01}
For every $\varepsilon>0$
\begin{eqnarray*}
\sup_{k\in\mathbb{Z}}\sup_{t\leq n,x\in[k.k+1]}|\eta_R(x,t)-\eta_R(k,t)|=o(n^{\frac{1}{4}+\varepsilon})~~~~a.s.
\end{eqnarray*}
\end{lem}

\begin{lem}\label{site-edge favor}
\begin{eqnarray*}
|\eta_R^*(n)-\xi_D^*(n)|=o(n^{\frac{1}{4}+\varepsilon})~~~~a.s.
\end{eqnarray*}
\end{lem}
\noindent {\bf Proof.} We have
\begin{eqnarray*}
\eta_R^*(n)-\xi_D^*(n)&=&\sup_{x\in\mathbb{R}}\eta_R(x,n)-\sup_{x\in\mathbb{Z}}\xi_D(x,n)\\
&=&\sup_{x\in\mathbb{R},y\in\mathbb{Z},|x-y|\leq 1}\left[\eta_R(x,n)-\xi_D(y,n)+\xi_D(y,n)\right]-\sup_{x\in\mathbb{Z}}\xi_D(x,n)\\
&\leq&\sup_{x\in\mathbb{R},y\in\mathbb{Z},|x-y|\leq 1}\left(\eta_R(x,n)-\xi_D(y,n)\right)+\sup_{y\in \mathbb{Z}} \xi_D(y,n)-\sup_{x\in\mathbb{Z}}\xi_D(x,n)\\
&=&\sup_{x\in\mathbb{R},y\in\mathbb{Z},|x-y|\leq 1}\left(\eta_R(x,n)-\xi_D(y,n)\right)\\
&\leq&\sup_{x\in\mathbb{R},y\in\mathbb{Z},|x-y|\leq 1}\left|\eta_R(x,n)-\xi_D(y,n)\right|.
\end{eqnarray*}
Similarly, we have
\begin{eqnarray*}
\xi_D^*(n)-\eta_R^*(n)\leq\sup_{x\in\mathbb{R},y\in\mathbb{Z},|x-y|\leq 1}\left|\eta_R(x,n)-\xi_D(y,n)\right|.
\end{eqnarray*}
Hence
\begin{eqnarray*}
\left|\eta_R^*(n)-\xi_D^*(n)\right|\leq\sup_{x\in\mathbb{R},y\in\mathbb{Z},|x-y|\leq 1}\left|\eta_R(x,n)-\xi_D(y,n)\right|.
\end{eqnarray*}
Then by Lemma \ref{edge-tool01} and Theorem \ref{One-side Invariance principle}, we get
\begin{eqnarray*}
&&|\eta_R^*(n)-\xi_D^*(n)|\\
&&\leq\sup_{x\in\mathbb{R},y\in\mathbb{Z},|x-y|\leq 1}|\eta_R(x,n)-\eta_R(y,n)+\eta_R(y,n)-\xi_D(y,n)|\\
&&\leq\sup_{x\in\mathbb{R},y\in\mathbb{Z},|x-y|\leq 1}|\eta_R(x,n)-\eta_R(y,n)|+\sup_{y\in\mathbb{Z}}|\eta_R(y,n)-\xi_D(y,n)|\\
&&=o(n^{\frac{1}{4}+\varepsilon}).
\end{eqnarray*}
The proof is complete.\hfill\fbox

\noindent {\bf Proof of Proposition \ref{pro-2.5}.}
By Theorem \ref{One-side Invariance principle} we can find a simple symmetric random walk $S_n$ and a Wiener process $W(t)$
on the same probability space such that for each $\varepsilon>0$,
\begin{eqnarray}\label{proof-pro-2.5-a}
\sup_{x\in\mathbb{Z}}|\xi_D(x,n)-\eta_R(x,n)|=o(n^{\frac{1}{4}+\varepsilon})~~~~~a.s.
\end{eqnarray}
For  any $\alpha>5$ and  $\varepsilon>0$, let $K_n=\max_{x\in\mathbb{Z},|x|\leq\sqrt{n}(\log n)^{-(2\alpha+1+\varepsilon)}}\xi_D(x,n)$. By
(\ref{proof-pro-2.5-a}), Lemmas \ref{inf-edge} and \ref{site-edge favor},
\begin{eqnarray*}
\xi_D^*(n)&\geq&\eta_R^*(n)-cn^{\frac{1}{4}+\varepsilon}\\
&\geq&I_R\left(\frac{\sqrt{n}}{(\log n)^{2\alpha+1+\varepsilon}},n\right)+\frac{1}{2}n^{\frac{1}{2}-\varepsilon}-cn^{\frac{1}{4}+\varepsilon}\\
&\geq&K_n+\frac{1}{2}n^{\frac{1}{2}-\varepsilon}-2cn^{\frac{1}{4}+\varepsilon}
>K_n
\end{eqnarray*}
for $n$ sufficiently large.
Thus the most visited edges of $S_n$ must be larger in absolute value than $\sqrt{n}(\log n)^{-(2\alpha+1+\varepsilon)}$ for $n$ large. For any $\gamma>11$,
choosing $\alpha$ and $\varepsilon$ so that $2\alpha+1+\varepsilon<\gamma$,
we obtain (\ref{pro-2.5-a}).  The proof is complete.
\hfill\fbox

\section{Preliminaries for the proof of Theorem \ref{mainthm}}

In this section, we make some preparations for the proof of Theorem \ref{mainthm}.
These preparations are modifications of the corresponding materials from \cite{DS18} and \cite{To01}.

\subsection{Three consecutive favorite edges}

We define the inverse edge local times by
\begin{eqnarray*}
T_U(x,k):=\min\{n\geq 1:\xi_U(x,n)=k\}\quad \mbox{and}\quad T_D(x,k):=\min\{n\geq 1:\xi_D(x,n)=k\}.
\end{eqnarray*}
For any $x\in \mathbb{Z}$,
define
\begin{eqnarray*}
u(x)&:=&\sum_{n=1}^{\infty}\textbf{1}_{\{S_{n-1}=x-2,S_n=x-1,x\in\mathcal{K}(n),\#\mathcal{K}(n)=3\}}\nonumber\\
&=&\sum_{k=1}^{\infty}\textbf{1}_{\{x\in\mathcal{K}(T_U(x-1,k)),
\#\mathcal{K}(T_U(x-1,k))=3\}}\nonumber\\
&=&\sum_{k=0}^{\infty}\sum_{h=1}^{\infty}
\textbf{1}_{\{x\in\mathcal{K}(T_U(x-1,k+1)),\#\mathcal{K}(T_U(x-1,k+1))=3,
L(x,T_U(x-1,k+1))=h\}}.
\end{eqnarray*}
Thus $f(3)\geq \sum_{x\in\mathbb{Z}}u(x)$.

For $x\in \mathbb{Z}$, $h, k\in \mathbb{N}$, we define
\begin{eqnarray}\label{3.1}
A_{x,h}^{(k)}&:=&\left\{\mathcal{K}(T_U(x-1,k+1))=\{x,x+1,x+2\},L(x,T_U(x-1,k+1))=h\right\}.
\end{eqnarray}
Here we use $T_U(x-1,k+1)$ instead of $T_U(x,k+1)$, which was used in \cite{DS18}, due to the the following two reasons:
\begin{itemize}
\item[(i)] {\it  If $x>0$, then by the definition of local time on the edge $x$ at
the time $n=T_U(x,k+1)$, the set of favorite edges $\mathcal{K}(n)$ is not equal to $\{x,x+1,x+2\}$, since $L(x+1,n)$ and $L(x+2,n)$ are even numbers and $L(x,n)$ is an odd number.}

\item[(ii)] {\it $T_U(x-1,k+1)$ is useful to obtain the lower bound on the first moment in Section 4.1.}
\end{itemize}

Note that the definition of $T_U(x-1,k+1)$ implies $S_{T_U(x-1,k+1)}=x-1$ and \linebreak $\xi_U\left(x-1,T_{U}(k+1,x-1)\right)=k+1$. Thus by \eqref{UDcrossings}, we have for $x\ge 1$
\begin{eqnarray*}
L(x-1,T_U(x-1,k+1))=2k+2-\textbf{1}_{\{0<x-1\}}.
\end{eqnarray*}
Hence
\begin{equation*}
L(x-1,T_U(x-1,k+1))=\left\{
\begin{aligned}
&2k+1, &\text{if~$x>1$},\\
&2k+2,
&\text{if~$x= 1$}.
\end{aligned}\right.
\end{equation*}
Again using \eqref{UDcrossings}, we can easily see that, when $x\ge 1$,
the $h$ in $A_{x,h}^{(k)}$ has to be even.
In the following, we implicitly assume $x>1$, which implies that $L(x-1,T_U(x-1,k+1))=2k+1$, unless explicitly mentioned otherwise.

We write the events $A_{x,h}^{(k)}$ in terms of $T_U(x-1,k+1)$ since the events
defined this way match
the form of the Ray-Knight representation to be discussed later. Let $K_h=\Big(\frac{1}{2}(h-2\sqrt{h}),\frac{1}{2}(h-\sqrt{h})\Big)$ and define
\begin{eqnarray}\label{N_H}
N_H:=\sum_{h=8}^{H}\sum_{k\in K_{2h}}\sum_{x=2}^{+\infty}\textbf{1}_{A_{x,2h}^{(k)}}\quad \mbox{and}\quad  N:=\lim_{H\to\infty}N_H=\sum_{h=8}^{+\infty}\sum_{k\in K_{2h}}\sum_{x=2}^{+\infty}\textbf{1}_{A_{x,2h}^{(k)}}.
\end{eqnarray}

We note that for each $h$, the events $A_{x,h}^{(k)},x\in\mathbb{Z},k\in K_h$ are mutually disjoint.
Since $f(3)\ge \sum_{x\in \mathbb{Z}}u(x)$, we have that $f(3)\geq N$, and thus it is enough to show that $N=\infty$ a.s.

\subsection{Branching process and the Ray-Knight representation}

In the remainder of this paper, we denote by $Y_n$ a critical Galton-Watson branching process with geometric offspring distribution, and by $Z_n,R_n$ two related critical Galton-Watson branching processes with immigration.
The precise definitions of these processes are as follows: Let $(X_{n, i})_{n, i}$
be i.i.d. geometric variables with mean 1, that is, for all $k\geq 0,P(X_{n,i}=k)=\frac{1}{2^{k+1}}$. We recursively define
\begin{eqnarray}\label{GWbp}
Y_{n+1}=\sum_{i=1}^{Y_n}X_{n,i},\ Z_{n+1}=\sum_{i=1}^{Z_n+1}X_{n,i},\ R_{n+1}=1+\sum_{i=1}^{R_n}X_{n,i}.
\end{eqnarray}
One can check that $Y_n,Z_n$ and $R_n$ are Markov chains with state space
$\mathbb{N}$ and transition probabilities:
\begin{equation}\label{TP}
P(Y_{n+1}=j|Y_n=i)=\pi(i,j):=\left\{
\begin{aligned}
&\delta_0(j), &\text{if~$i=0$},\\
&2^{-i-j}\frac{(i+j-1)!}{(i-1)!j!}, &\text{if~$i>0$},
\end{aligned}\right.
\end{equation}
\begin{eqnarray}
P(Z_{n+1}=j|Z_n=i)=\rho(i,j):=\pi(i+1,j),\\
P(R_{n+1}=j|R_n=i)=\rho^{*}(i,j):=\pi(i,j-1).
\end{eqnarray}

Let $k\geq 0$ and $x$ be fixed integers. When $x-1\geq 1$, we define the following three processes:

$1$, $(Z_n^{(k)})_{n\geq 0}$ is a Markov chain with transition probabilities $\rho(i,j)$ and initial state $Z^{(k)}_0=k$.

$2$, $(Y_n^{(k)})_{n\geq -1}$ is a Markov chain with transition probabilities $\pi(i,j)$ and initial state $Y^{(k)}_{-1}=k$.

$3$, $(Y_n^{'(k)})_{n\geq 0}$ is a Markov chain with transition probabilities $\pi(i,j)$ and initial state $Y_{0}^{'(k)}=Z_{x-1}^{(k)}$.

\noindent We assume that the three processes are independent, except that $Y^{'(k)}$ starts from
$Z_{x-1}^{(k)}$.
We patch the above three processes together to a single process as follows:
\begin{equation}\label{RK1}
\Delta_x^{(k)}(y):=\left\{
\begin{array}{ll}
Z_{x-1-y}^{(k)}, &\text{if~$0\leq y\leq x-1$},\\
Y_{y-x}^{(k)}, &\text{if~$x-1\leq y\leq \infty$},\\
Y_{-y}^{'(k)},&\text{if~$-\infty<y\leq 0$}.
\end{array}\right.
\end{equation}
By the Ray-Knight theorem on local times of simple random walk on $\mathbb{Z}$
(c.f. \cite[Theorem~1.1]{Kn63}),
we know that for any integers $x\geq 2$ and $k\geq 0$,
\begin{eqnarray}\label{3.6}
(\xi_D(y,T_U(x-1,k+1)),y\in\mathbb{Z})\overset{{\rm law}}{=}(\Delta_{x-1}^{(k)}(y),y\in\mathbb{Z}).
\end{eqnarray}

Similarly, when $x-1\leq 0$, we define the processes:

$1$, $(R_n^{(k)})_{n\geq -1}$ is a Markov chain with transition probabilities $\rho^*(i,j)$ and initial state $R^{(k)}_{-1}=k$.

$2$, $(Y_n^{(k)})_{n\geq 0}$ is a Markov chain with transition probabilities $\pi(i,j)$ and initial state $Y^{(k)}_{0}=k$.

$3$, $(Y_n^{'(k)})_{n\geq -1}$ is a Markov chain with transition probabilities $\pi(i,j)$ and initial state $Y_{-1}^{'(k)}=R_{-1-x}^{(k)}$.

\noindent
We assume that the three processes are independent, except that $Y^{'(k)}$ starts from $R_{-1-x}^{(k)}$.
In this case, we patch the three processes together as follows:
\begin{equation}\label{RK2}
\Delta_x^{(k)}(y)\triangleq\left\{
\begin{array}{ll}
Y_{y}^{'(k)}, &\text{if~$-1\leq y< \infty$},\\
R_{y-x}^{(k)}, &\text{if~$x-1\leq y\leq -1$},\\
Y_{x-1-y}^{(k)},&\text{if~$-\infty<y\leq x-1$}.
\end{array}\right.
\end{equation}
By the Ray-Knight theorem, we get that for the case $k\geq 0,x\leq 1$:
\begin{eqnarray}
(\xi_D(y,T_U(x-1,k+1)),y\in\mathbb{Z})\overset{{\rm law}}{=}(\Delta_{x-1}^{(k)}(y),y\in\mathbb{Z}).
\end{eqnarray}

\subsection{Three favorite edges under Ray-Knight representation}

For $h\in\mathbb{N}$, define the first hitting times
of $[h,+\infty)$ for $Y_n^{(k)}$ and $Z_n^{(k)}$ to be $\sigma_{h}^{(k)}$ and $\tau_h^{(k)}$ respectively, and the extinction time of $Y_n^{(k)}$ to be $\omega^{(k)}$. That is,
\begin{eqnarray}\label{tingshi}
\sigma_h^{(k)}:=\min\{n\geq 0: Y_{n}^{(k)}\geq h\},\tau_h^{(k)}:=\min\{n\geq 0:Z_n^{(k)}\geq h\},\omega^{(k)}:=\min\{n\geq 0:Y_n^{(k)}=0\}.\quad
\end{eqnarray}
Using the notation above, we can express $P(A_{x,h}^{(k)})$ in its Ray-Knight representation form. In the remainder of this section, we let
$\tilde{n}:=T_U(x-1,k+1)$ for simplicity.
Now, on $A_{x,h}^{(k)}$,  we have that
$L(x,\tilde{n})=L(x+1,\tilde{n})=L(x+2,\tilde{n})=h$, $h>L(y,\tilde{n})=2\xi_D(y-1,\tilde{n})+\textbf{1}_{\{0<y\leq x-1\}}$ for
$x\ge 2,$ $y\neq x,x+1,x+2$. We have the following five cases:

(1) $y\leq 0,~\xi_D(y-1,\tilde{n})=\frac{L(y,\tilde{n})}{2}<\frac{h}{2}$;

(2) $y\in [1,x-2],~\xi_D(y-1,\tilde{n})=\frac{L(y,\tilde{n})-1}{2}<\frac{h-1}{2}$;

(3) $y=x-1,~\xi_D(y-1,\tilde{n})=\frac{L(y,\tilde{n})-1}{2}<\frac{h-1}{2}$;

(4) $y=x,x+1,x+2,~\xi_D(y-1,\tilde{n})=\frac{L(y,\tilde{n})}{2}=\frac{h}{2}$;

(5) $y\geq x+3,~\xi_D(y-1,\tilde{n})=\frac{L(y,\tilde{n})}{2}<\frac{h}{2}.$\\

Then by (\ref{3.6}), we obtain that
\begin{eqnarray}\label{3.11}
P(A_{x,h}^{(k)})&=&P\left(Y_0^{(k)}=Y_1^{(k)}=Y_2^{(k)}=\frac{h}{2},\left
\{Y_n^{(k)}<\frac{h}{2},n\geq 3\right\},\right.\nonumber \\
&&\quad\quad\left.\left\{Z_n^{(k)}<\frac{h-1}{2},1\leq n\leq x-2\right\},\left\{Y_n^{'(k)}<\frac{h}{2},n\geq 1\right\}\right).
\end{eqnarray}

For all the notation above, when the initial state of a process is obvious, we omit the superscript $``(k)"$ for simplicity. We will also use conditional probability $P(\cdot|Y_0=k)$ to indicate the initial state.

\subsection{Standard lemmas}

In this subsection we recall a few lemmas that will be useful later. In what follows, $c_i$ for $i\geq 1$ and $c$ are all constants.

\begin{lem}(\cite[Lemma~2.2]{DS18})\label{Lem01}
We have that

$(1)$ For $i,j\in\Big(\frac{1}{2}(h-10\sqrt{h}),\frac{1}{2}(h+10\sqrt{h})\Big)$, there exists positive constants $c$ and $C$ such that $ch^{-\frac{1}{2}}\leq \pi(i,j)\leq Ch^{-\frac{1}{2}}$ for all $h\geq 100$.

$(2)$ For $i+j=h$, $\pi(i,j)\leq O(1)h^{-\frac{1}{2}}$.

$(3)$ For $j<i_1<i_2$, $\pi(i_1,j)>\pi(i_2,j)$.
\end{lem}

\begin{lem}(\cite[Lemma~2.3]{DS18})\label{Lem02}
For any $h\in\mathbb{N}$, it holds that
$E\tau_h=EZ_{\tau_h}-Z_0$.
In particular, for any $0\le k\le h$,
we have that $E[\tau_h|Z_0=k]\geq h-k$.
\end{lem}

\begin{lem}(\cite[(6.18)]{To01})\label{Lem03}
There exists a constant $C<\infty$ such that, for any $0\leq k<h$,
\begin{eqnarray*}
E(Z_{\tau_h}|Z_0=k)\leq h+Ch^{\frac{1}{2}}.
\end{eqnarray*}
\end{lem}

\section{Proof of Theorem \ref{mainthm}}

In this section, we give the proof of Theorem \ref{mainthm} by following
the idea of \cite{DS18}.  We spell out some details for the reader's convenience and point out some modifications that need to be made.

\subsection{Lower bound on the first moment.}

The following is the counterpart to \cite[Lemma 3.1]{DS18}.

\begin{lem}\label{Lem-First Moment}
Suppose that $Z_0=k\in K_h=[\frac{h-2\sqrt{h}}{2},\frac{h-\sqrt{h}}{2}]$. Then there exists a constant $c>0$ such that for
any $h>4$,
$$E\left(\sum_{n=1}^{\tau_{\frac{h-1}{2}}}\frac{h/2-Z_n}{h/2}\right)\geq c\sqrt{h}.$$
\end{lem}
\noindent {\bf Proof.} Let $M_n=\sum_{s=1}^{n}(Z_s-s)-n(Z_n-n)$, and let $\mathcal{F}_n=\sigma(Z_0,Z_1,...,Z_n)$. By the proof of \cite[Lemma 3.1]{DS18}, we know that $(M_n)$ is a martingale. By the optional stopping theorem, we get that
$$ E\left(\sum_{n=1}^{\tau_{\frac{h-1}{2}}}(Z_n-n)\right)=
E\tau_{\frac{h-1}{2}}(Z_{\tau_{\frac{h-1}{2}}}-\tau_{\frac{h-1}{2}}),
$$
and thus
\begin{eqnarray}\label{4.1}
E\left(\sum_{n=1}^{\tau_{\frac{h-1}{2}}}\frac{h/2-Z_n}{h/2}\right)
&=&E\tau_{\frac{h-1}{2}}-\dfrac{E\left(\sum_{n=1}^{\tau_{\frac{h-1}{2}}}(Z_n-n)+n\right)}
{h/2}\nonumber\\
&=&E\tau_{\frac{h-1}{2}}-\dfrac{E\tau_{\frac{h-1}{2}}(Z_{\tau_{\frac{h-1}{2}}}-
\tau_{\frac{h-1}{2}})}{h/2}-\dfrac{E\frac{(1+\tau_{\frac{h-1}{2}})
\tau_{\frac{h-1}{2}}}{2}}{h/2}\nonumber\\
&=&\left(1-\frac{1}{h}\right)E\tau_{\frac{h-1}{2}}-\frac{2}{h}
E\left[\tau_{\frac{h-1}{2}}Z_{\tau_{\frac{h-1}{2}}}-\frac{1}{2}\tau_{\frac{h-1}{2}}^2\right].
\end{eqnarray}
Define the process $M_n'=-\frac{1}{4}Z_n^2+nZ_n-\frac{1}{2}n^2+\frac{1}{4}n$. By the proof of \cite[Lemma 3.1]{DS18}, we know that $(M_n')$ is a martingale.
Applying the optional stopping theorem to $(M_n')$ at $\tau_{\frac{h-1}{2}}$,
we get
\begin{eqnarray}\label{4.2}
E\left[\tau_{\frac{h-1}{2}}Z_{\tau_{\frac{h-1}{2}}}-\frac{1}{2}\tau_{\frac{h-1}{2}}^2\right]
=E\left[\frac{1}{4}Z_{\tau_{\frac{h-1}{2}}}^2-\frac{1}{4}
\tau_{\frac{h-1}{2}}\right]-\frac{Z_0^2}{4}
=\frac{1}{4}E\left[Z_{\tau_{\frac{h-1}{2}}}^2-\tau_{\frac{h-1}{2}}\right]-\frac{Z_0^2}{4}.
\end{eqnarray}
Combining \eqref{4.1}, \eqref{4.2} and Lemma \ref{Lem02}, we get
\begin{eqnarray}
&&E\left(\sum_{n=1}^{\tau_{\frac{h-1}{2}}}\frac{h/2-Z_n}{h/2}\right)
=\left(1-\frac{1}{2h}\right)E\tau_{\frac{h-1}{2}}-\frac{1}{2h}
\left(EZ_{\tau_{\frac{h-1}{2}}}^2-Z_0^2\right)\nonumber\\
&&= \left(1-\frac{1}{2h}\right)E\left[Z_{\tau_{\frac{h-1}{2}}}-Z_0\right]-
\frac{1}{2h}E\left[(Z_{\tau_{\frac{h-1}{2}}}-Z_0)(Z_{\tau_{\frac{h-1}{2}}}+Z_0)\right]\nonumber\\
&&= \frac{1}{2h}E\left[\left(2h-1-(Z_{\tau_{\frac{h-1}{2}}}+Z_0)\right)
(Z_{\tau_{\frac{h-1}{2}}}-Z_0)\right].
\end{eqnarray}
Obviously $Z_{\tau_{\frac{h-1}{2}}}-Z_0\geq \frac{h-1}{2}-\frac{h-\sqrt{h}}{2}\geq c_1\sqrt{h}$. Then by Lemmas \ref{Lem02} and \ref{Lem03}, we obtain
\begin{eqnarray}
E\left(\sum_{n=1}^{\tau_{\frac{h-1}{2}}}\frac{h/2-Z_n}{h/2}\right)\geq \frac{1}{2h}\cdot\left(2h-1-\frac{h-1}{2}-c_2\sqrt{h}-\frac{h-\sqrt{h}}{2}\right)\cdot c_1\sqrt{h}\geq c\sqrt{h}.
\end{eqnarray}\hfill\fbox

\smallskip

\begin{pro}\label{First Moment}
There exists $c>0$ such that $EN_H\geq c\log H$ for all
$H\in [50,\infty)$.
\end{pro}

\noindent {\bf Proof.} By the Ray-Knight representation, (\ref{N_H}) and (\ref{3.11}), we know that
\begin{eqnarray*}
EN_H&=&
\sum_{h=8}^{H}\sum_{k\in K_{2h}}P\left(Y_0^{(k)}=Y_1^{(k)}=Y_2^{(k)}=h,\left\{Y_n^{(k)}<h,\forall n\geq 3\right\}\right)\\
&&\cdot
\sum_{x=2}^{+\infty}P\left(\left\{Z_n^{(k)}<h-\frac{1}{2},1\leq n\leq x-2\right\},\left\{Y_n^{'(k)}<h, \forall n\geq 1\right\}\right).
\end{eqnarray*}
It follows that
\begin{eqnarray}\label{Pro first 01}
&&EN_H\nonumber\\
&\geq &
\sum_{h=50}^{H}\sum_{k\in K_{2h}}P\left(Y_0^{(k)}=Y_1^{(k)}=Y_2^{(k)}=h,\left\{Y_3^{(k)}\in \left(h-5\sqrt{2h}, h-\frac{\sqrt{2h}}{2}\right),Y_n^{(k)}<h,\forall n\geq 4\right\}\right)\nonumber\\
&&\cdot
\sum_{x=2}^{+\infty}P\left(\left\{Z_n^{(k)}<h-\frac{1}{2},1\leq n\leq x-2\right\},\left\{Y_n^{'(k)}<h, n\geq 1\right\}\right)\nonumber\\
&=&
\sum_{h=50}^{H}\sum_{k\in K_{2h}}\pi(k,h)\cdot\pi(h,h)
\cdot\pi(h,h)\cdot\sum_{m\in(h-5\sqrt{2h},h-\frac{\sqrt{2h}}{2})}
\pi(h,m) P\left(Y_n^{(m)}<h,\forall n\geq 1\right)\nonumber\\
&&\cdot
\sum_{x=2}^{+\infty}P\left(\tau_{h-\frac{1}{2}}\geq x-1,\left\{Y_n^{'(k)}<h,\forall n\geq 1\right\}\right).
\end{eqnarray}
By Lemma \ref{Lem01}, all the $\pi(\cdot,\cdot)$'s in the display above are of the order $h^{-\frac{1}{2}}$. Since $Y_n$ is a martingale,
applying the optional stopping theorem at $\sigma_{h}\land \omega$,
where $\sigma_{h}$ and $\omega$ are defined in \eqref{tingshi}, we
get for $m\in(h-5\sqrt{2h},h-\frac{\sqrt{2h}}{2})$,
\begin{eqnarray}\label{4.6}
P\left(Y_n^{(m)}<h,\forall n\geq 1\right)=P\left(Y_n^{(m)}~\mbox{hits }0\mbox{ before exceeds }h\right)\geq \dfrac{h-m}{h}\geq c_3h^{-\frac{1}{2}}.
\end{eqnarray}
The last two inequalities hold since
\begin{eqnarray*}
m&=&
EY_{\sigma_{h}\land\omega}=E\left(Y_{\sigma_{h}}
\textbf{1}_{\{\omega\geq\sigma_{h}\}}\right)
+E\left(Y_{\omega}\textbf{1}_{\{\omega<\sigma_{h}\}}\right)\\
&\geq&E\left(Y_{\sigma_{h}}
\textbf{1}_{\{\omega\geq\sigma_{h}\}}\right)\\
&\geq&
h\left(1-P\left(Y_n^{(m)}~\mbox{hits }0\mbox{ before exceeds }h\right)\right),
\end{eqnarray*}
which implies that
\begin{eqnarray}\label{4.7}
P\left(Y_n^{(m)}~\mbox{hits }0\mbox{ before exceeds }h\right)\geq \dfrac{h-m}{h}\geq c_3h^{-\frac{1}{2}}.
\end{eqnarray}
By (\ref{Pro first 01}) and (\ref{4.6}), we get
\begin{eqnarray}\label{Pro first 02}
EN_H\geq c_4
\sum_{h=50}^{H}\sum_{k\in K_{2h}}h^{-2}\cdot \sum_{x=2}^{+\infty}P\left(\tau_{h-\frac{1}{2}}\geq x-1,\left\{Y_n^{'(k)}<h,\forall n\geq 1\right\}\right).
\end{eqnarray}
By the independence of the processes in the Ray-Knight representation, we have
\begin{eqnarray*}
&&\sum_{x=2}^{+\infty}
P\left(\tau_{h-\frac{1}{2}}\geq x-1,\left\{Y_n^{'(k)}<h,\forall n\geq 1\right\}\right)\\
&&\ge\sum_{x=2}^{+\infty}
\sum_{l=0}^{h-1}P\left(Z_n^{(k)}<h-\frac{1}{2}\mbox{ for } 1\leq n\leq x-2,Z_{x-1}=l\right)\cdot P\left(Y_n^{(l)}~\mbox{hits }0\mbox{ before exceeds }h\right).
\end{eqnarray*}
By the definitions of $(Z_n)_{n\ge 0}$ and $\tau_h$, we obtain
\begin{eqnarray}\label{Pro first 02-a}
Z_0-Z_{\tau_{h-\frac{1}{2}}}\leq h-\frac{\sqrt{2h}}{2}-h+\frac{1}{2}=\frac{1-\sqrt{2h}}{2}\leq 0.
\end{eqnarray}
Similar to (\ref{4.7}), we have
$P(Y_n^{(l)}~\mbox{hits }0\mbox{ before exceeds }h)\geq \frac{h-l}{h}$,
which together with (\ref{Pro first 02-a}) and  Lemma \ref{Lem-First Moment} implies that
\begin{eqnarray}\label{Pro first 03}
&&\sum_{x=2}^{+\infty}
P\left(\tau_{h-\frac{1}{2}}\geq x-1,\left\{Y_n^{'(k)}<h,\forall n\geq 1\right\}\right)\nonumber\\
&&\geq\sum_{x=2}^{+\infty}
\sum_{l=0}^{h}P\left(\tau_{h-\frac{1}{2}}\geq x-1,Z_{x-2}=l\right)\cdot\frac{h-l}{h}\nonumber\\
&&=
E\left(\sum_{l=0}^{h}\sum_{x=2}^{\tau_{h-\frac{1}{2}}+1}
\frac{h-l}{h}\cdot\textbf{1}_{\{Z_{x-2}=l\}}\right)
=E\left(\sum_{n=0}^{\tau_{h-\frac{1}{2}}-1}\frac{h-Z_n}{h}\right)\nonumber\\
&&\ge
E\left(\sum_{n=0}^{\tau_{h-\frac{1}{2}}-1}\frac{h-Z_n}{h}\right)+
\left(\frac{h-Z_{\tau_{h-\frac{1}{2}}}}{h}-\frac{h-Z_0}{h}\right)\nonumber\\
&&=
E\left(\sum_{n=1}^{\tau_{h-\frac{1}{2}}}\frac{h-Z_n}{h}\right)
\geq c_5\sqrt{h}.
\end{eqnarray}
By (\ref{Pro first 02}) and (\ref{Pro first 03}), we obtain
\begin{eqnarray*}
EN_H\geq c_4
\sum_{h=50}^{H}\sum_{k\in K_{2h}}h^{-2}\cdot c_5\sqrt{h}\geq c_6\sum_{h=100}^{H}h^{-1}\geq c\log H.
\end{eqnarray*}
The proof is complete.\hfill\fbox

\subsection{Upper bound on the second moment}

In this subsection, we will give an upper bound on the second moment $EN_H^2$ following \cite{DS18}. If we use the $N_H$ defined in \eqref{N_H}--\eqref{3.1},
we could not prove the counterparts of \cite[Lemmas 3.3, 3.4]{DS18}.
To overcome this, we give a variant $\widetilde{N}_H$ of $N_H$ with $N_H\le \widetilde{N}_H$ a.s. In the remainder of this section, we assume $x\ge 2$ unless
explicitly mentioned otherwise.

For any positive integer $h$, we define
\begin{eqnarray}\label{4.2-a}
A_{x,2h}&:=&\left\{\mathcal{K}\left(T_D(x-1,h)\right)=\{x,x+1,x+2\},
L\left(x,T_D(x-1,h)\right)=2h\right\}
\end{eqnarray}
and
\begin{eqnarray*}
\widetilde{N}_H:=\sum_{h=50}^{H}\sum_{x=2}^{\infty}\textbf{1}_{A_{x,2h}},
\quad
\widetilde{N}:=\lim_{H\to\infty}\widetilde{N}_H=
\sum_{h=50}^{\infty}\sum_{x=2}^{\infty}\textbf{1}_{A_{x,2h}}.
\end{eqnarray*}

We claim that
\begin{eqnarray}\label{4.2-b}
\cup_{k\in K_{2h}}A_{x,2h}^{(k)}\subset A_{x,2h},
\end{eqnarray}
where $A_{x,2h}^{(k)}$ is defined by (\ref{3.1}).
Suppose that $\omega\in A_{x,2h}^{(k)}$ for some $k\in K_{2h}$. Define
$$
\tilde{T}(\omega):=\{m<T_U(x-1,k+1)(\omega): S_m=x-1,S_{m-1}=x,\mathcal{K}(m)=\{x,x+1,x+2\},
L(x,m)=2h\}.
$$
By the definition of $A_{x,h}^{(k)}$, we know that
$\tilde{T}(\omega)$ consists of one unique element $t(\omega)$.
Further by (\ref{UDcrossings}), we get that
$t(\omega)=T_D(x-1, h)(\omega)$. Hence $\omega\in A_{x,2h}$ and thus (\ref{4.2-b}) holds.

By (\ref{4.2-b}), we get that $N_H\leq\widetilde{N}_H$ a.s.
Now we study the second moment of $\widetilde{N}_H$ following  \cite[Section 3.2]{DS18}.

Let $D(n)=\big(\xi_D(x, n),x\in\mathbb{Z}\big)\in\mathbb{N}^{\mathbb{Z}}$ be the random vector that records the number of downcrossings of each site by time $n$.
For $l\in \mathbb{N}^{\mathbb{Z}}$, we use $l(i),i\in\mathbb{Z}$ to denote the $i$-th component of $l$. For $l\in \mathbb{N}^{\mathbb{Z}}$, define $B_x(l)=\{\exists n<\infty:D(n)=l,S(n-1)=x,S(n)=x-1\}$. Note that if $B_x(l)$ happens, there exists a unique $n\in\mathbb{N}$, such that $D(n)=l,S(n-1)=x$ and $S(n)=x-1$.

Let $\mathcal{P}=\{l:P(B_x(l))>0~\mbox{for some}~x\}$. For $\mathcal{Q}\subset\mathcal{P}$, denote $B_x(\mathcal{Q})=\bigcup_{l\in\mathcal{Q}}B_x(l)$. Then by virtue of (\ref{UDcrossings}),  we have
$A_{x,2h}=B_x(\mathcal{P}_{x,2h})$, where $\mathcal{P}_{x,2h}$ is the set  of $l\in\mathcal{P}$ such that
\begin{eqnarray*}
l(x-1)=l(x)=l(x+1)=
h,\ l(i)<h\ \mbox{for all}\ i\neq x-1,x,x+1.
\end{eqnarray*}

Let $\mathcal{A}$ be the family of all subsets of $\mathcal{P}$. For any $x\in\mathbb{Z}$, we define a map $\varphi_x:\mathcal{P}\mapsto\mathcal{A}$ by
\begin{eqnarray*}
\varphi_x(l):=\{l^*\in\mathcal{P}:l^*(i)<l(i)~\mbox{for}~i=x,x+1,l^*(i)=l(i)~\mbox{for}~i\neq x,x+1\}.
\end{eqnarray*}

Following the argument of \cite[Lemma 3.3]{DS18}, we can prove

\begin{lem}\label{Second Moment01}
Suppose $x_1, x_2\in \mathbb{Z}$ and $h$ is a positive integer.
If $l_i^*\in \varphi_{x_i}(l_i)$ with $l_i\in\mathcal{P}_{x_i,2h}$, $i=1,2$,
we have that $B_{x_1}(l_1^*)\cap B_{x_2}(l_2^*)=\emptyset$, if $(x_1,l_1)\neq(x_2,l_2)$. Further, we have $B_{x_1}(l_1^*)\cap B_{x_2}(l_2^*)=\emptyset$ if $(x_1,l_1)=(x_2,l_2)$ but $l_1^*\neq l_2^*$.
\end{lem}

The following result is the counterpart of \cite[Lemma 3.4]{DS18}.

\begin{lem}\label{Second Moment02}
There exists a constant $c>0$ such that for any
$x\ge 2, h\ge 50, l\in\mathcal{P}_{x,2h}$,
\begin{eqnarray*}
P\big(B_x(\varphi_x(l))\big)\geq chP\big(B_x(l)\big).
\end{eqnarray*}
\end{lem}
\noindent {\bf Proof.} We consider
$l^*\in \varphi_x(l)$ such that $l^*(x)\in [h-\frac{\sqrt{2h}}{2},h)$ and $l^*(x+1)\in [h-\frac{\sqrt{2h}}{2},h)$.
According to Lemma \ref{Lem01} $(1)$ and $(3)$, there is a constant $c>0$ such that
\begin{eqnarray*}
\dfrac{P\big(B_x(l^*)\big)}{P\big(B_x(l)\big)}=
\dfrac{\pi\big(h,l^*(x)\big)\cdot\pi\big(l^*(x),l^*(x+1)\big)\cdot\pi\big(l^*(x+1),l(x+2)\big)}
{\pi(h,h)\cdot\pi(h,h)\cdot\pi(h,l(x+2))}
\geq 4c.
\end{eqnarray*}
Note that
there are about $h/2$ of such $l^*\in\varphi_x(l)$
that satisfy the above inequality. Thus
by Lemma \ref{Second Moment01}, we get that
$
P\big(B_x(\varphi_x(l))\big)\geq chP\big(B_x(l)\big).
$\hfill\fbox

\smallskip

With Lemmas \ref{Second Moment01} and \ref{Second Moment02} in hand,
we can follow the proof of \cite[Proposition 3.5]{DS18} to get the following result. We omit the details.

\begin{pro}\label{Second Moment}
We have that $E\widetilde{N}_H^2=O(\log H)\cdot E\widetilde{N}_H$.
\end{pro}

\begin{cor}\label{cor-4.6}
We have $EN_H^2=O(\log^2 H).$
\end{cor}
\noindent {\bf Proof.} By Proposition \ref{Second Moment} and the Cauchy-Schwarz inequality, we get
$$
E\widetilde{N}_H^2\leq O(\log H)\cdot \left(E\widetilde{N}_H^2\right)^{1/2},
$$
which implies that
$$
E\tilde{N}_H^2=O(\log^2 H).
$$
Since $N_H\leq \tilde{N}_H$ a.s., the desired result follows immediately.
\hfill\fbox

\subsection{0-1 law}

Recall that $N=\lim_{H\to \infty}N_H$. First, we show that $N=\infty$ with positive probability.

\begin{pro}\label{pro-4.7}
There exists a constant $\delta>0$ such that
$P(N=\infty)\geq\delta$.
\end{pro}
\begin{proof}
By the Cauchy-Schwarz inequality, we get that
\begin{eqnarray}\label{pro-4.7-a}
EN_H&=&EN_H\textbf{1}_{\{N_H>\log\log H\}}+EN_H\textbf{1}_{\{N_H\leq\log\log H\}}\nonumber\\
&\leq&\sqrt{EN_H^2\cdot P(N_H>\log\log H)}+\log\log H.
\end{eqnarray}
Combining this with Proposition \ref{First Moment} and Corollary \ref{cor-4.6},
we get that there exist constants $c$, $\delta>0$ such that
\begin{eqnarray*}
P(N_H>\log\log H)\geq \dfrac{(EN_H-\log\log H)^2}{EN_H^2}\geq c\dfrac{\log^2H}{EN_H^2}\geq\delta,
\end{eqnarray*}
for all sufficiently large $H$. Letting $H\to\infty$, we get that $P(N=\infty)\geq\delta$.
\end{proof}

Recall the definition of $\tilde{U}(n)$ in Theorem \ref{thm-2.1}.
By Theorem \ref{thm-2.1}, we know that
\begin{eqnarray}\label{Favorite edge infty}
\liminf_{n\to\infty}\dfrac{\tilde{U}(n)}{\sqrt{n}(\log n)^{-\gamma}}=\infty
\qquad a.s.,
\end{eqnarray}
where $\gamma>11$.

Using Proposition \ref{pro-4.7} and (\ref{Favorite edge infty}), following the argument of \cite[Section 3.3]{DS18} and applying Kolmogorov's 0-1 law, we obtain that
$$
P(f(3)=\infty)=1,
$$
which completes the proof of Theorem \ref{mainthm}.

\section{Remark}

In Section 2,  we used the transience of the favorite downcrossing site process to show the transience of the favorite edge process. In fact, from Proposition \ref{pro-2.4},
we can see that there is a close relation
between favorite edges and favorite downcrossing sites. A natural question arise:

{\it How about the number of favorite downcrossing sites of one-dimensional simple symmetric random walk?}

In \cite{HHMS22}, we will consider this question and prove the following result.

\begin{thm}
For a  one-dimensional simple symmetric random walk,
with probability 1 there are only finitely many times at which there are at least four favorite downcrossing sites and three favorite downcrossing sites occurs infinitely often.
\end{thm}

\bigskip

{ \noindent {\bf\large Acknowledgments}\ \  This work was supported by the National Natural Science Foundation of China
(Grant Nos. 12171335, 12101429,  12071011, 11931004 and  11871184), the Simons Foundation (\#429343) and the Science Development Project of Sichuan University (2020SCUNL201).

\end{document}